\documentclass[a4paper,12pt]{amsart}
\usepackage[left=1in,right=1in,top=1.2in,bottom=1.2in]{geometry}

\usepackage{times}
\usepackage{microtype}

\usepackage{commath,amssymb,amscd,mathrsfs,mathtools,dsfont,accents,tikz-cd}
\usepackage[all]{xy}
\usepackage{enumerate}

\usepackage[pagebackref,linktocpage]{hyperref}
\hypersetup{
    pdftitle={Iitaka K3},            
    pdfauthor={Hu Fei},     
    colorlinks=true,        
    linkcolor=red,          
    citecolor=blue,         
    filecolor=magenta,      
    urlcolor=cyan           
}

\newtheorem{thm}{Theorem}[section]
\newtheorem{lemma}[thm]{Lemma}

\newtheorem{cor}[thm]{Corollary}

\theoremstyle{definition}

\newtheorem{remark}[thm]{Remark}

\theoremstyle{remark}
\newtheorem*{rmk}{Remark}

\newcommand{\cA}{\mathcal{A}}

\newcommand{\cF}{\mathcal{F}}

\newcommand{\cO}{\mathcal{O}}

\newcommand{\bA}{\mathbb{A}}
\newcommand{\bC}{\mathbb{C}}

\newcommand{\bG}{\mathbf{G}}

\newcommand{\bZ}{\mathbb{Z}}
\newcommand{\bP}{\mathbb{P}}
\newcommand{\bQ}{\mathbb{Q}}

\newcommand{\Pic}{\operatorname{Pic}}

\newcommand{\aff}{\operatorname{aff}}

\newcommand{\alb}{\operatorname{alb}}

\newcommand{\Aut}{\operatorname{Aut}}

\newcommand{\Ker}{\operatorname{Ker}}
\newcommand{\NE}{\overline{\operatorname{NE}}}

\newcommand{\NS}{\operatorname{NS}}

\newcommand{\rank}{\operatorname{rank}}
\newcommand{\Sing}{\operatorname{Sing}}

\newcommand{\Alb}{\operatorname{Alb}}

\newcommand{\isom}{\simeq}
\newcommand{\ratmap}{\dashrightarrow}

\begin{document}

\title[The dimension of automorphism groups of algebraic varieties]
{The dimension of automorphism groups of algebraic varieties with pseudo-effective log canonical divisors}

\author{Fei Hu}
\address{\textsc{Department of Mathematics} \endgraf \textsc{National University of Singapore} \endgraf \textsc{10 Lower Kent Ridge Road, Singapore 119076}}
\email{\href{mailto:hf@u.nus.edu}{\tt hf@u.nus.edu}}


\begin{abstract}
Let $(X,D)$ be a log smooth pair of dimension $n$, where $D$ is a reduced effective divisor such that the log canonical divisor $K_X + D$ is pseudo-effective.
Let $G$ be a connected algebraic subgroup of $\Aut(X,D)$.
We show that $G$ is a semi-abelian variety of dimension $\le \min\{n-\bar{\kappa}(V), n\}$ with $V\coloneqq X\setminus D$.
In the dimension two, Iitaka claimed in \cite{Iitaka79} that $\dim G\le \bar{q}(V)$ for a log smooth surface pair with $\bar{\kappa}(V) = 0$ and $\bar{p}_g(V) = 1$.
We (re)prove and generalize this classical result for all surfaces with $\bar{\kappa}=0$
without assuming Iitaka's classification of logarithmic Iitaka surfaces or logarithmic $K3$ surfaces.
\end{abstract}

\subjclass[2010]{
14J50, 
14L10, 
14L30. 
}

\keywords{automorphism, semi-abelian variety, group action, logarithmic Kodaira dimension}

\thanks{}

\maketitle


%
%
%
%

\section{Introduction} \label{section_intro}

\noindent
Throughout this paper, unless otherwise stated, we work over the field $\bC$ of complex numbers.
Let $V$ be an algebraic variety.
By Nagata, there is a complete algebraic variety $\overline{V}$ containing $V$ as a Zariski-dense open subvariety.
Then by Hironaka, there exist a {\it log smooth} pair $(X,D)$,
i.e., $X$ is a smooth projective variety and $D$ is a reduced effective divisor with only simple normal crossing (SNC) singularities,
and a projective birational morphism $\pi\colon X\to \overline{V}$ such that $D = \pi^{-1}(\overline{V}\setminus V)$ and $X\setminus D = \pi^{-1}(V)$.
Such a pair is called a {\it log smooth completion} of $V$ and $D$ is called the {\it boundary divisor}.
We then define
$$\text{the \it logarithmic irregularity } \bar{q}(V) \coloneqq h^0(X, \Omega_X^1(\log D)), $$
$$\text{the \it logarithmic geometric genus } \bar{p}_g(V) \coloneqq h^0(X, K_X + D), $$
$$\text{the \it logarithmic Kodaira dimension } \bar{\kappa}(V) \coloneqq \kappa(X, K_X + D), $$
where $\Omega_X^1(\log D)$ is the logarithmic differential sheaf,
$h^i(-)$ denotes the complex dimension of $H^i(-)$ and $\kappa$ denotes the Iitaka $D$-dimension.
It is known that these numerical invariants are independent of the choice of the log smooth completion $(X,D)$.
See \cite[\S 11]{Iitaka82} for details.

Let $G$ be a connected algebraic group.
By Chevalley's structure theorem on algebraic groups,
there exists a unique connected affine normal subgroup $G_{\rm aff}$ of $G$
such that the quotient group $A_G\coloneqq G/G_{\rm aff}$ is an abelian variety.
Moreover, the quotient morphism is the Albanese morphism $\alb_G$ of $G$.
If $G_{\rm aff}$ is an algebraic torus, denoted as $T_G$, then $G$ is called a {\it semi-abelian variety},
i.e., there is an exact sequence of connected algebraic groups:
\begin{equation} \label{eq_semi_abelian} \tag{$\spadesuit$}
1\longrightarrow T_G \longrightarrow G \xrightarrow{\alb_G} A_G \longrightarrow 1.
\end{equation}
It is known that such $G$ is a commutative algebraic group (see e.g. \cite[Proposition 3.1.1]{BSU13}).

Due to Serre \cite[Th\'eor\`emes 5 and 7]{Serre58},
there exist an abelian variety $\Alb(V)$ (resp. a semi-abelian variety $\cA_V$)
and a morphism $\alb_V \colon V \to \Alb(V)$ (resp. a morphism $\alpha_V\colon V \to \cA_V$)
such that any morphism from $V$ to an abelian variety (resp. a semi-abelian variety) factors,
uniquely up to translations, through this $\Alb(V)$ (resp. $\cA_V$).
Then $\Alb(V)$ (resp. $\alb_V$) is called the {\it Albanese variety} (resp. the {\it Albanese morphism}) of $V$,
and $\cA_V$ (resp. $\alpha_V$) is called the {\it quasi-Albanese variety} (resp. the {\it quasi-Albanese morphism}) of $V$.
Note, however, that this construction of the Albanese morphism is, in general, not of birational nature.
Alternatively, one can {\it birationally} define the Albanese variety and the Albanese map (which is only a rational map; cf.~\cite[Chapter II, \S3]{Lang83}).
See \cite[Th\'eor\`eme 6]{Serre58} for the relation between these two definitions.
From the viewpoint of birational geometry, they are the same in characteristic zero for normal projective varieties with only rational singularities (cf.~\cite[Lemma 8.1]{Kawamata85}).

Let $V$ be a smooth algebraic variety with some log smooth completion $(X,D)$
obtained by blowing up subvarieties of the boundary such that $V = X\setminus D$.
Then the Albanese varieties of $V$ and $X$ are isomorphic to each other and the Albanese morphism $\alb_V$ of $V$ is just the restriction of the Albanese morphism $\alb_X$ of $X$.
Also, the Albanese morphism $\alb_V$ of $V$ factors through the quasi-Albanese morphism $\alpha_V$ of $V$.
That is, we have the following commutative diagram:
\[\xymatrix{
\cA_V \ar@{->}[drr]_-{\alb_{\cA_V}} & & V \ar@{->}[d]^-{\alb_V} \ar@{->}[ll]_-{\alpha_V} \ar@{^{(}->}[r]^-{j} & X \ar@{->}[d]^-{\alb_X}  \\
& & \Alb(V) \ar@{->}[r]_-{\alb(j)}^-{\isom} & \Alb(X).
}\]
Further, the quasi-Albanese variety $\cA_V$ of $V$ can be constructed using the space of logarithmic $1$-forms $H^0(X, \Omega_X^1(\log D))$.
See \cite{Fujino15, Iitaka76} for more details about this construction, which depends on Deligne's mixed Hodge theory for smooth complex algebraic varieties (unlike Serre's construction \cite{Serre58}, which is valid over an algebraically closed field of arbitrary characteristic).
It is known that $\dim \cA_V = \bar q(X)$ and $\dim \Alb(V) = q(X)\coloneqq h^1(X, \cO_X)$.
If we assume further that $V$ is projective, then the quasi-Albanese morphism $\alpha_V$ of $V$ is just the original Albanese morphism $\alb_V$ of $V$.

We shall refer to \cite{KM98} for the standard definitions, notations, and terminologies in birational geometry.
For instance, see \cite[Definitions 2.34, 2.37, and 5.8]{KM98} for the definitions of {\it canonical} singularity,
Kawamata log terminal singularity ({\it klt}), divisorial log terminal singularity ({\it dlt}), log canonical singularity ({\it lc}), and {\it rational singularity}.

\begin{thm} \label{ThmA}
Let $(X,D)$ be a projective $\bQ$-factorial dlt pair of dimension $n$
and $D$ a reduced effective divisor such that $K_X + D$ is pseudo-effective.
Let $\Aut(X,D)$ denote the stabilizer of the boundary $D$ (viewed as a subset of $X$) in the automorphism group $\Aut(X)$ of $X$.
Let $G$ be a connected algebraic subgroup of $\Aut(X,D)$.
Then the following assertions hold.
\begin{enumerate}[{\em (1)}]
  \item \label{ThmA_1} $G$ is a semi-abelian variety sitting in the exact sequence \eqref{eq_semi_abelian} of dimension at most
$$\min \! \left\{ n - \kappa(X, K_X + D), n \right\} \! .$$
  \item \label{ThmA_2} When $\dim G = n$, $X$ is a $G$-equivariant compactification of $G$ such that $K_X + D \sim 0$.
  \item  Suppose further that $\kappa(X, K_X + D)\ge 0$. Then we have
\begin{enumerate}[{\em (a)}]
  \item \label{ThmA_3a} $\dim G\le n$ and the equality holds only if $\kappa(X, K_X + D) = 0$ and the dimension of the abelian variety $A_G$ equals $q(X)$;
  \item \label{ThmA_3b} $\dim T_G\le n$ and the equality holds only if $\kappa(X, K_X + D) = 0$ and $\dim A_G = q(X) = 0$.
\end{enumerate}
\end{enumerate}
\end{thm}

A {\it logarithmic Iitaka surface} is a smooth algebraic surface $V$ such that the logarithmic Kodaira dimension $\bar{\kappa}(V) = 0$
and the logarithmic geometric genus $\bar{p}_g(V) = 1$.
In this case by Kawamata \cite[Corollary 29]{Kawamata81}, we know that the logarithmic irregularity $\bar{q}(V)\le \dim V = 2$.
If assume further that $\bar{q}(V) = 0$, we then call $V$ a {\it logarithmic $K3$ surface}. See \cite{Iitaka79} for details.

Next, we (re)prove and generalize \cite[Theorem 5]{Iitaka79} in which
Iitaka provided an upper bound of the dimension of automorphism groups of certain logarithmic Iitaka surfaces.
However, his (implicit) proof depends heavily on his classification of logarithmic Iitaka surfaces and logarithmic $K3$ surfaces, so that we are not able to follow his proof completely.
Here we offer a classification-free proof for all smooth surfaces with vanishing logarithmic Kodaira dimension.

\begin{thm} \label{ThmD}
Let $(X,D)$ be a log smooth pair of dimension $2$ with $V\coloneqq X\setminus D$ such that $\bar{\kappa}(V) = 0$.
Let $G$ be a connected algebraic subgroup of $\Aut(X,D)$.
Then $G$ is a semi-abelian variety of dimension at most $\bar{q}(V)$.
If assume further that $\bar{p}_g(V) = 0$, then $\dim G\le q(X)$.
\end{thm}

\begin{rmk} \label{rmk_Brion}
It is known that for an abelian variety $A$ acting faithfully on a smooth algebraic variety $X$,
the induced group homomorphism $A\to \Alb(X)$ has a finite kernel by the Nishi--Matsumura theorem (cf.~\cite{Matsumura63}).
In particular, we have $\dim A\le \dim \Alb(X)=q(X)$.
However, for a semi-abelian variety $G$ acting faithfully on a smooth algebraic variety $V$,
by Brion's example\footnote{
The author is grateful to Professor Michel Brion for a conversation about his (counter)example.}
below one cannot try to prove $G\to \cA_V$ has a finite kernel and to deduce $\dim G\le \bar{q}(V)$.

Let $X$ be the projective plane $\bP^2$ and $D$ the union of a smooth conic and a transversal line.
In homogeneous coordinates, one can take for $D$ the union of $(xy = z^2)$ and $(z = 0)$.
Then the neutral component of the automorphism group of $(X,D)$ is a one-dimensional algebraic torus,
acting via $t\cdot [x:y:z] = [tx:t^{-1}y:z]$.
Also, $V\coloneqq X \setminus D$ is the complement of the conic $(xy = 1)$ in the affine plane $\bA^2$ with coordinates $x,y$.
So the quasi-Albanese variety $\cA_V$ of $V$ is a one-dimensional algebraic torus too,
and the quasi-Albanese morphism $\alpha_V$ is just given by $xy - 1$
(which generates the group of all invertible regular functions on $V$ modulo constants).
Then $\alpha_V$ is $G$-invariant and hence $G$ does not act on $\cA_V$ with a finite kernel.
\end{rmk}

The following two corollaries are direct consequences of our main theorems, Sumihiro's equivariant completion theorem (cf.~\cite[Theorem 3]{Sumihiro74}), and the equivariant resolution theorem (see \cite[Proposition 3.9.1 and Theorem 3.36]{Kollar07} for a modern description).
Indeed, let $V$ be a normal algebraic variety and $G$ a {\it linear} algebraic subgroup of $\Aut(V)$.
Sumihiro's theorem asserts that there exists a $G$-equivariant completion $\overline{V}$ of $V$.
Let $(X,D)$ be a $G$-equivariant resolution of singularities of $\overline{V}$.
Thus we may identify $G$ with a subgroup of $\Aut(X,D)$ so that our main theorems apply.

\begin{cor} \label{CorA}
Let $V$ be a normal algebraic variety of logarithmic Kodaira dimension $\bar{\kappa}(V)\ge 0$,
and $G$ a connected linear algebraic subgroup of $\Aut(V)$.
Then $G$ is an algebraic torus of dimension at most $\min\{\dim V-\bar{\kappa}(V), \dim V\}$.
\end{cor}

\begin{cor} \label{CorB}
Let $V$ be a smooth algebraic surface, and $(X,D)$ a log smooth completion such that $V = X\setminus D$.
Suppose that $\bar{\kappa}(V) = 0$ and let $G$ be a connected linear algebraic subgroup of $\Aut(V)$.
Then $G$ is an algebraic torus of dimension at most $\bar{q}(X)$.
If assume further that $\bar{p}_g(V) = 0$, then $\dim G\le q(X)$.
\end{cor}



%
%
%
%

\section{Proof of Theorem \ref{ThmA}} \label{section_proof_1}

We first prove that $G$ is a semi-abelian variety (see \eqref{eq_semi_abelian} for its definition and related notations),
under a slightly weaker condition than that of Theorem \ref{ThmA}.

We remark that $G$ is a semi-abelian variety if and only if $G$ does not contain any algebraic subgroup isomorphic to the one-dimensional additive algebraic group $\bG_a$.
In fact, by Chevalley's structure theorem, to show $G$ is a semi-abelian variety, it suffices to show that the affine normal subgroup $G_{\aff}$ of $G$ is an algebraic torus.
Consider the unipotent radical $R_u(G_{\aff})$ of $G_{\aff}$.
If it is not trivial, then it contains $\bG_a$. So we may assume that $G_{\aff}$ is reductive.
Note that any non-trivial semi-simple subgroup of $G_{\aff}$ also contains $\bG_a$.
Thus by the structure theory of reductive groups, $G_{\aff} = R(G_{\aff})$ is an algebraic torus.

\begin{lemma} \label{lemma_semi_abelian}
Let $(X,D)$ be a projective log canonical pair and $D$ a reduced effective divisor such that $K_X + D$ is pseudo-effective.
Let $G$ be a connected algebraic subgroup of $\Aut(X,D)$.
Then $G$ is a semi-abelian variety.
\end{lemma}
\begin{proof}
Take a $G$-equivariant log resolution $\pi\colon \widetilde{X}\to X$ of the pair $(X,D)$.
Then we may write
\begin{equation*} 
K_{\widetilde{X}} + \widetilde{D} = \pi^*(K_X + D) + \sum a_iE_i,
\end{equation*}
where $\widetilde{D}\coloneqq \pi^{-1}_*D + E$ with $\pi^{-1}_*D$ the strict transform of $D$ and $E\coloneqq \sum E_i$ the sum of all $\pi$-exceptional divisors.
Note that for every $E_i$, the log discrepancy $a_i \coloneqq 1 + a(E_i,X,D)$ is non-negative, since $(X,D)$ is log canonical.
Thus $K_{\widetilde{X}} + \widetilde{D}$ is also pseudo-effective.
Moreover, $G$ is a subgroup of $\Aut(\widetilde{X},\widetilde{D})$ since $\pi$ is a $G$-equivariant log resolution.
Therefore, replacing $(X,D)$ by $(\widetilde{X},\widetilde{D})$, we may assume that $(X,D)$ is log smooth.

Suppose to the contrary that $G$ contains some algebraic subgroup isomorphic to $\bG_a$.
Consider the faithful action of $\bG_a$ on $(X,D)$.
It is a generically free action since $\bG_a$ admits no non-trivial algebraic subgroup.
More precisely, outside the closed subset $F$ of all fixed points of $\bG_a$-action, this action is free,
i.e., the $\bG_a$-orbit of any point $x\in X\setminus F$ is isomorphic to the affine line via the orbit map.
Thus we obtain a dominating family of rational curves on $X$ by completing these $\bG_a$-orbits.
A general rational curve (not contained in $D$) of this family can only intersect the boundary $D$ in at most one point.
Note that if a proper variety is dominated by rational curves, then it is in fact covered by rational curves (cf.~\cite[Corollary 1.4.4]{Kollar96}).
Hence by \cite[Lemma 2.1]{BZ15}, it follows that $K_X + D$ is not pseudo-effective which contradicts our assumption.
\end{proof}

Next, we give an upper bound of the dimension of a semi-abelian variety acting faithfully on an arbitrary algebraic variety.

\begin{lemma} \label{lemma_max_dim}
Let $G$ be a semi-abelian variety.
Suppose that $G$ acts faithfully on an algebraic variety $V$ of dimension $n$. Then we have
$$\dim G\le \min \! \left\{n - \bar{\kappa}(V), n \right\} \! .$$
In particular, if $\dim G = n$, then $V$ contains a Zariski open orbit with trivial isotropy group.
\end{lemma}
\begin{proof}
Let $T_G$ denote the algebraic torus as in the definition of the semi-abelian variety $G$.
Then $T_G$ acts generically freely on $V$ by \cite[\S 1.6, Corollaire~1]{Demazure70}.
In other words, there exists a Zariski open subvariety $U$ of $V$ such that the isotropy group $(T_G)_x$ is trivial for any $x\in U$.
Note that the isotropy group $G_x$ has a fixed point $x$ and hence is affine by \cite[Proposition~2.1.6]{BSU13}.
Thus the neutral component of $G_x$ is contained in $(T_G)_x$, so is trivial for any $x\in U$.
Therefore, $G_x$ is finite for any $x\in U$.
Then we can easily get
$$n = \dim V \ge \dim G\cdot x = \dim G - \dim G_x = \dim G.$$

Suppose that $\dim G = n = \dim V$.
Then for any $x\in U$, the orbit $G\cdot x$ is Zariski-dense in $V$.
Equivalently, since every orbit is locally closed, $G\cdot x$ is a Zariski open subvariety of $V$.
Note that the isotropy group $G_x$ acts trivially on $G\cdot x$ because $G$ is commutative, so does $G_x$ on $V$.
This implies that $G_x$ is trivial since the whole $G$-action is faithful.
Thus in this optimal case, we have proved the assertion in the lemma.

On the other hand, by a theorem due to Rosenlicht (cf.~\cite[Theorem 2]{Rosenlicht56}),
there exists a Zariski open subvariety $V_0$ of $V$ such that the geometric quotient $V_0/G$ exists.
Consider the natural quotient map $V_0 \to V_0/G$ with a general fibre $F = G\cdot x_0$ for some $x_0\in U\cap V_0$.
By Iitaka's easy addition formula (cf.~\cite[Theorem 11.9]{Iitaka82}), we have
$$\bar{\kappa}(V) \le \bar{\kappa}(V_0) \le \bar{\kappa}(F) + \dim(V_0/G).$$
Note that this general fibre $F$ is isomorphic to $G/G_{x_0}$, where $G_{x_0}$ is finite as $x_0\in U$.
Thus $F$ is also a semi-abelian variety and hence $\bar{\kappa}(F) = 0$.
By a dimension formula for quotient varieties, we have
$$\dim(V_0/G) = \dim V - \dim G + \underset{x\in V}{\min} \dim G_x = n - \dim G.$$
Combining the last two displayed (in)equalities, we prove that $\dim G \le n - \bar{\kappa}(V)$.
Together with $\dim G\le n$ we just proved, we obtain the desired upper bound of $\dim G$.
\end{proof}

\begin{rmk}
Without the condition $\dim G = n$, one can still show that the semi-abelian variety $G$ acts generically freely on $V$.
Actually, the isotropy group $G_x$ is generically finite by the proof of Lemma \ref{lemma_max_dim}.
Note that finite subgroups of a semi-abelian variety form a countable family.
So one may use \cite[Lemma 5]{Iitaka77} to conclude that $G_x$ is generically trivial.
\end{rmk}

\begin{proof}[Proof of Theorem \ref{ThmA}]
Replacing $(X,D)$ by some $G$-equivariant log resolution as in the proof of Lemma \ref{lemma_semi_abelian},
we may assume that $(X,D)$ is log smooth with pseudo-effective $K_X + D$,
and $V \coloneqq X\setminus D$ is a smooth quasi-projective subvariety of $X$.
Note that after this replacement, the Iitaka dimension $\kappa(X, K_X + D)$ remains the same equal to $\bar{\kappa}(V)$.
In Lemma \ref{lemma_semi_abelian} we have proved that $G$ is a semi-abelian variety.
We may then regard $G\le \Aut^0(X,D)$ as an algebraic subgroup of $\Aut(V)$.
Indeed, the natural restriction map $G \to G|_{V}$ is an isomorphism.
Thus applying Lemma \ref{lemma_max_dim} to the faithful $G$-action on $V$, the assertion (\ref{ThmA_1}) follows.

For the assertion (\ref{ThmA_2}), by the above lemma again,
$X$ (actually $V$) contains a Zariski open orbit $V'$ with trivial isotropy group (so $V'\isom G$).
Let $D'\coloneqq X\setminus V'$ be the (total) boundary of this almost homogeneous variety.
By taking a $G$-equivariant log resolution, we may assume that $D'$ is a simple normal crossing divisor containing $D$.
Thus by \cite[Theorem 1.1]{BZ15}, we have $K_X + D'\sim 0$.
But $K_X + D$ is already pseudo-effective.
So $D=D'$ and hence $K_X + D\sim 0$.
Since the push-forward of a linearly trivial divisor is also linearly equivalent to zero,
the original pair $(X,D)$ has trivial log canonical divisor.
This shows the assertion (\ref{ThmA_2}).

Suppose that $\kappa(X, K_X + D)\ge 0$ and $\dim G=n$.
Then $\bar{\kappa}(V) = \kappa(X, K_X + D) = 0$.
We now show the second equality in assertion (\ref{ThmA_3a}), i.e., $\dim A_G = q(X)$.
First it follows from \cite[Theorem~28]{Kawamata81} that the quasi-Albanese morphism $\alpha_V\colon V\to \cA_V$ is an open algebraic fibre space
(i.e., generically surjective with irreducible general fibres).
In particular,
$$\bar{q}(V) = \dim \cA_V \le \dim V = n.$$
By the universal properties of the Albanese morphism $\alb_V$ and the quasi-Albanese morphism $\alpha_V$,
we know that the $G$-action on $V$ descends uniquely to an action of $G$ on the abelian variety $\Alb(V)$ and the semi-abelian variety $\cA_V$, respectively.
In other words, we have the following two exact sequences of connected algebraic groups:
$$1 \longrightarrow K_A \longrightarrow G \longrightarrow \Aut^0(\Alb(V))=\Alb(V),$$
$$1 \longrightarrow K_\alpha \longrightarrow G \longrightarrow \Aut^0(\cA_V),$$
where $K_A$ and $K_\alpha$ denote the corresponding kernels.
On the other hand, both $G$ and $\cA_V$ are semi-abelian,
so we also have the following two exact sequences of connected algebraic groups:
$$1\longrightarrow T_G \longrightarrow G \xrightarrow{\alb_G} A_G \longrightarrow 1,$$
$$1\longrightarrow T_{\cA_V} \longrightarrow \cA_V \xrightarrow{\alb_{\cA_V}} \Alb(\cA_V) = \Alb(V) \longrightarrow 1.$$

By the Nishi--Matsumura theorem (cf.~\cite{Matsumura63}), the induced group homomorphism $G\to \Alb(V)$ factors through $A_G$
such that the group homomorphism $A_G\to \Alb(V)$ has a finite kernel.
In particular, we have
\begin{equation*}
\dim A_G\le \dim \Alb(V)=q(X).
\end{equation*}
Identify the torus $T_G/T_G\cap K_\alpha$ with its image in $\Aut^0(\cA_V)$.
Note that the induced action of $T_G$ on the abelian variety $\Alb(V)$ is trivial.
Thus $T_G/T_G\cap K_\alpha$ acts faithfully on $T_{\cA_V}$.
Note that a torus acting faithfully on another torus must act by multiplication.
Set $d\coloneqq \dim T_G$. We have
\begin{equation*}
d - \dim T_G\cap K_\alpha = \dim T_G/T_G\cap K_\alpha \le \dim T_{\cA_V} \eqqcolon t.
\end{equation*}
Note that the torus $T_G\cap K_\alpha$ acts trivially on $\cA_V$,
and hence it acts faithfully on the general fibre $F$ of the quasi-Albanese morphism $\alpha_V\colon V\to \cA_V$.
By \cite[\S 1.6, Corollaire 1]{Demazure70} we have
\begin{equation*}
\dim T_G\cap K_\alpha \le \dim F = \dim V - \dim \cA_V = n - \bar{q}(V).
\end{equation*}
In order to satisfy $\dim G = d+\dim A_G = n = q(X)+t+(n-\bar{q}(V))$,
all of the last three displayed inequalities should be equalities.
In particular, we have $\dim A_G = q(X)$.

Note that $(X,D)$ is a projective dlt pair and hence $X$ has only rational singularities (cf.~\cite[Theorem 5.22]{KM98}).
For such $X$, its irregularity $q(X)$ does not depend on its resolution.
Thus $\dim \Alb(V)$ equals $q(X)$ of the original $X$.
This completes the proof of the assertion (\ref{ThmA_3a}).

For the last assertion (\ref{ThmA_3b}), if the dimension of the torus part $T_G$ of $G$ is maximal (i.e., $d = n$),
then we have $\dim G = n$ and hence all statements in the assertion (\ref{ThmA_3a}) hold.
Moreover, it follows from $G = T_G$ that $A_G = 0$.
Thus $q(X) = \dim A_G = 0$. (In particular, $\cA_V$ is an algebraic torus of dimension $\bar{q}(V) = t\le n$.)
We have completed the proof of Theorem \ref{ThmA}.
\end{proof}

\section{Proof of Theorem \ref{ThmD}} \label{section_proof_2}

In this section, we will prove Theorem \ref{ThmD} which shall be divided into Theorems \ref{ThmB} and \ref{ThmC}.
The first theorem considers the logarithmic Iitaka surfaces (i.e., $\bar{\kappa} = 0$ and $\bar{p}_g = 1$),
while all other algebraic surfaces with $\bar{\kappa} = \bar{p}_g = 0$ are dealt with by the second one.
We first prepare some general results used to prove both theorems.

\subsection{Preliminaries}

We will frequently and implicitly use the following lemma to compare logarithmic invariants of an algebraic variety with an open subvariety.

\begin{lemma}[{cf.~\cite[Proposition 11.4]{Iitaka82}}] \label{lemma_open_subvar}
Let $X$ be an algebraic variety and $U$ a nonempty Zariski open subvariety of $X$. Then we have
\[
\bar{q}(X) \le \bar{q}(U),\ \bar{p}_g(X)\le \bar{p}_g(U)\ \text{and}\ \bar{\kappa}(X) \le \bar{\kappa}(U).
\]
\end{lemma}

The following lemma is a slight generalization of \cite[Lemma 2]{Iitaka79} which is used to compute $\bar{q}$.

\begin{lemma} \label{Iitaka_lemma_q_bar}
Let $(X,D)$ be a log smooth pair of dimension $n$ with $V\coloneqq X\setminus D$.
Let $\sum D_i$ be the irreducible decomposition of $D$.
Then we have
$$\bar{q}(V) = q(X) + \rank \Ker(\oplus \bZ [D_i] \to \NS(X)),$$
where $\NS(X)\coloneqq \Pic(X)/\Pic^0(X)$ denotes the N\'eron--Severi group of $X$.
\end{lemma}
\begin{proof}
We first have the following long exact sequence of local cohomology:
$$\cdots \to H^1_D(X, \bZ) \to H^1(X, \bZ) \to H^1(V, \bZ) \to H^2_D(X, \bZ) \to H^2(X, \bZ) \to \cdots,$$
where $H^i_D(X, \bZ)$ are the cohomology groups of $X$ with support in $D$.
Since $D$ is of codimension $1$ in $X$, we get that $H^i_D(X, \bZ) = 0$ for $i = 0,1$ (cf.~\cite[Lemma 23.1]{milneLEC}).
It is also known that
$$H^2_D(X, \bZ) \isom \oplus H^2_{D_i}(X, \bZ) \isom \oplus H^0(D_i, \bZ) = \oplus \bZ[D_i],$$
as the first isomorphism follows from the Mayer--Vietoris sequence and the second one holds by Poincar\'e--Lefschetz duality.
Note that $\dim_{\bC} H^1(V, \bC) - \dim_{\bC} H^1(X, \bC) = (\bar{q}(V) + q(X)) - 2q(X) = \bar{q}(V) - q(X)$.
On the other hand, the preceding exact sequence yields
$$H^1(V, \bZ)/H^1(X, \bZ) \isom \Ker(\oplus \bZ [D_i] \to H^2(X, \bZ)) = \Ker(\oplus \bZ [D_i] \to \NS(X)).$$
So we obtain the lemma.
\end{proof}

As a corollary of the above lemma and the well-known behavior of N\'eron--Severi groups under a point blowup,
we can readily see that the logarithmic irregularity $\bar{q}$ is (somewhat) invariant under any point blowup.

\begin{lemma} \label{lemma_q_bar}
Let $(X,D)$ be a log smooth pair of dimension $n$ with $V\coloneqq X\setminus D$.
Let $\pi\colon \widetilde{X}\to X$ be the blowup of some point.
Let $\widetilde{D}$ be the sum of the strict transform $\pi^{-1}_*D$ and some reduced effective exceptional divisors.
Denote $\widetilde{V}\coloneqq \widetilde{X}\setminus \widetilde{D}$.
Then we have $\bar{q}(\widetilde{V}) = \bar{q}(V)$. \qed
\end{lemma}

For a log surface, we also need the following formula to calculate $\bar{p}_g$ which is due to Sakai (cf.~\cite[Lemma 1.12]{Sakai80}).
See also \cite[Chapter 1, Lemma 2.3.1]{Miyanishi01} for a similar treatment.

\begin{lemma} \label{lemma_p_g_bar}
Let $(X,D)$ be a log smooth pair of dimension $2$ with $V\coloneqq X\setminus D$.
Then we have
$$\bar{p}_g(V) = p_g(X) + h^1(\cO_D) - q(X) + \gamma(D),$$
where $\gamma(D) \coloneqq \dim \Ker \{ H^1(X, \cO_X) \to H^1(D, \cO_D) \}$.
In particular, if $X$ is a regular surface, then $\bar{p}_g(V) = p_g(X) + h^1(\cO_D)$.
\end{lemma}

\begin{remark} \label{rmk_h1}
For a log smooth surface pair $(X,D)$, if $D$ is connected such that $h^1(\cO_D) = 1$,
then $D$ contains a smooth elliptic curve or a cycle of smooth rational curves.
Indeed, recall that the {\it arithmetic genus} of the divisor $D$ is defined as $p_a(D) \coloneqq 1-\chi(\cO_D) = h^1(\cO_D)$,
where $\chi(\cF)\coloneqq \sum (-1)^i h^i(\cF)$ denotes the {\it Euler characteristic} of a coherent sheaf $\cF$.
By the Riemann--Roch theorem, one has the following genus formula:
$$p_a(D) = 1 + \frac{1}{2}D.(K_X + D).$$
Then it is easy to see that the arithmetic genus of a tree of smooth rational curves is zero.
Meanwhile, if $D$ contains a curve of genus greater than one, then $p_a(D)\ge 2$.

For a general (not necessarily connected) boundary $D$, the above genus formula also holds.
By the induction on the number of the connected components of $D$,
one can show that
$$h^1(\cO_D) = \sum h^1(\cO_{D^j}),$$
where each $D^j$ is a connected component of $D$.
Therefore, if we assume that $h^1(\cO_D) = 1$,
then there exists a unique connected component of $D$ such that its arithmetic genus is one.
We know the behavior of this connected component from the discussion above.
\end{remark}

\subsection{Logarithmic Iitaka surfaces}

In this subsection, we will prove Theorem \ref{ThmD} in the setting of the title which means $\bar{\kappa}=0$ and $\bar{p}_g=1$.
We actually prove the following theorem.

\begin{thm} \label{ThmB}
Let $V$ be a logarithmic Iitaka surface, and $(X, D)$ a log smooth completion such that $V = X\setminus D$.
Then the neutral component of $\Aut(X, D)$ is a semi-abelian variety of dimension at most $\bar{q}(V)$.
\end{thm}

By the classification theory of smooth projective surfaces,
we can show that for a logarithmic Iitaka surface $V$,
the ambient surface $X$ can only have the following possibilities.
In this paper, a {\it ruled surface} always means a birationally ruled surface, i.e., birationally equivalent to $C\times \bP^1$ for some smooth curve $C$.
The genus of the base curve $C$ is also called the genus of the ruled surface.
In particular, an {\it elliptic ruled surface} is a ruled surface of genus $1$.

\begin{lemma} \label{lemma_log_Iitaka}
Let $V$ be a logarithmic Iitaka surface, and $(X, D)$ a log smooth completion such that $V = X\setminus D$.
Then $X$ belongs to one of the following cases:
\begin{enumerate}[{\em (1)}]
  \item \label{lemma_log_Iitaka1} $X$ is a rational surface; $p_g(X)=q(X)=0$ and $h^1(\cO_D)=1$.
  \item \label{lemma_log_Iitaka2} $X$ is an elliptic ruled surface; $p_g(X)=0$ and $q(X)=1$.
  \item \label{lemma_log_Iitaka3} $X$ is (birationally) a $K3$ surface or an abelian surface; $p_g(X)=1$.
\end{enumerate}
\end{lemma}
\begin{proof}
By Lemma \ref{lemma_open_subvar}, one has $\kappa(X)\le \bar{\kappa}(V) = 0$ and $p_g(X)\le \bar{p}_g(V) = 1$.
We first consider the case $\kappa(X) < 0$.
If $q(X)=0$, this is just the case (\ref{lemma_log_Iitaka1}).
If $q(X)\ge 1$, we will show that $X$ cannot be a ruled surface of genus $q(X)> 1$.
Let $\alb_X\colon X\to B\subseteq \Alb(X)$ be the Albanese morphism of $X$ with $B$ a smooth projective curve of genus $q(X)\ge 1$ because $p_g(X)=0$.
For a general point $b\in B$, the fibre $F\coloneqq\alb_X^{-1}(b)$ of $b$ is a smooth rational curve.
Define $F_V \coloneqq F\cap V$.
Then by Iitaka's easy addition formula (cf.~\cite[Theorem 11.9]{Iitaka82}), we have
$$0 = \bar{\kappa}(V) \le \bar{\kappa}(F_V) + \dim B,$$
which implies $\bar{\kappa}(F_V) \ge 0$.
On the other hand, by Kawamata's addition formula (for morphisms of relative dimension one, cf.~\cite[Theorem 11.15]{Iitaka82}), we have
$$0 = \bar{\kappa}(V) \ge \bar{\kappa}(F_V) + {\kappa}(B)\ge {\kappa}(B)\ge 0.$$
Thus $\bar{\kappa}(F_V) = {\kappa}(B) = 0$ and hence $B$ is an elliptic curve.
This gives us the case (\ref{lemma_log_Iitaka2}).

We next consider the case $\kappa(X) = 0$.
To obtain the case (\ref{lemma_log_Iitaka3}), we just need to rule out the Enriques surfaces and the hyperelliptic surfaces.
If $X$ is (birationally) an Enriques surface, then there exists a finite \'etale cover $\sigma\colon\widetilde{X}\to X$ for some (birationally) $K3$ surface $\widetilde{X}$.
Let $\widetilde{D}\coloneqq \sigma^{-1}D$ and $\widetilde{V}\coloneqq \widetilde{X}\setminus \widetilde{D}$.
Then we have $\bar{\kappa}(\widetilde{V}) = \bar{\kappa}(V) = 0$
since $\sigma|_{\widetilde{V}}\colon \widetilde{V}\to V$ is also a finite \'etale cover (cf.~\cite[Theorem 11.10]{Iitaka82}).
On the other hand, it follows from Lemma \ref{lemma_p_g_bar} that $h^1(\cO_D)=1$
and hence there exists a unique connected component $D^1$ of $D$ such that $h^1(\cO_{D^1})=1$ (see also Remark \ref{rmk_h1}).
Take a connected component $\widetilde{D}^{1}$ of $\widetilde{D}$ which maps onto $D^1$.
Then consider the induced finite \'etale cover $\widetilde{D}^1 \to D^1$.
Say its degree is $d$. We obtain the following equalities:
$$1-h^1(\cO_{\widetilde{D}^1}) = \chi(\cO_{\widetilde{D}^1}) = d \cdot \chi(\cO_{D^1}) = d \cdot (1-h^1(\cO_{D^1})) = 0.$$
Hence by Lemma \ref{lemma_p_g_bar} again we have
$\bar{p}_g(\widetilde{V}) = p_g(\widetilde{X}) + h^1(\cO_{\widetilde{D}}) \ge p_g(\widetilde{X}) + h^1(\cO_{\widetilde{D}^1}) = 2$.
This contradicts the fact $\bar{\kappa}(\widetilde{V}) = 0$.

If $X$ is (birationally) a hyperelliptic surface, then there exists a minimal model $X_m$ of $X$ obtained by contracting all $(-1)$-curves.
Let $\mu\colon X\to X_m$ denote the composite contraction morphism and $D_m\coloneqq \mu_*D$.
There is an effective divisor $R_{\mu}$ with support the union of all $\mu$-exceptional divisors
such that $K_X=\mu^*K_{X_m}+R_\mu$ by the ramification formula (cf.~\cite[Theorem 5.5]{Iitaka82}).
Then we have
\begin{align*}
0 &= \bar{\kappa}(V) = \kappa(X, K_X + D) = \kappa(X,\mu^*K_{X_m}+R_{\mu}+D) \\
&= \kappa(X,\mu^*K_{X_m}+ N R_{\mu}+D) \text{ for } N\gg 0 \text{ by \cite[Lemma 10.5]{Iitaka82}}\\
&\ge \kappa(X,\mu^*K_{X_m}+ \mu^*D_m) = \kappa(X_m,K_{X_m}+D_m).
\end{align*}
Note that $K_{X_m}\sim_{\bQ} 0$ and $D_m$ is nef (see e.g. \cite[Lemma 2.11]{HTZ16}).
So we have
$$0 \ge \kappa(X_m, K_{X_m}+D_m) = \kappa(X_m,D_m) \ge 0,$$
and hence $D_m \sim_{\bQ} 0$ by the abundance theorem for surfaces (cf.~\cite[Theorem 6.2]{Fujino12}).
Then $D_m$ being effective implies that $D_m=0$, i.e., $D$ is $\mu$-exceptional.
Thus by the projection formula,
\begin{align*}
H^0(X, K_X + D) = H^0(X,\mu^*K_{X_m}+R_{\mu}+D) 
= H^0(X_m,K_{X_m}) = 0.
\end{align*}
This contradicts the assumption $\bar{p}_g(V) = h^0(X, K_X + D) = 1$.
\end{proof}

\begin{remark} \label{rmk_D_A}
Let $V$ be a logarithmic Iitaka surface, and $(X, D)$ a log smooth completion such that $V = X\setminus D$.
Let $\sum D_i$ be the irreducible decomposition of $D$.
We would like to introduce an associated divisor $D_A$ separately for each case in Lemma \ref{lemma_log_Iitaka} as follows.
\begin{enumerate}[(1)]
  \item \label{rmk_D_A_1} If $X$ is a rational surface, by Remark \ref{rmk_h1} there are two subcases:
  \begin{enumerate}[(i)]
    \item \label{rmk_D_A_1i} $D_i$ is a smooth elliptic curve for some $i$, then let $D_A = D_i$;
    \item \label{rmk_D_A_1ii} $p_a(D_i) = 0$ for all $i$,
    then there is a cycle of smooth rational curves $\sum_{i=1}^r\! D_i \eqqcolon \! D_A$.
  \end{enumerate}
  \item \label{rmk_D_A_2} If $X$ is an elliptic ruled surface, we have seen in the proof of Lemma \ref{lemma_log_Iitaka} that
  the general fibre $F_V$ (of the restriction morphism $\alb_X|_V\colon V\to B$) is rational and has logarithmic Kodaira dimension zero.
  So it is isomorphic to the one-dimensional algebraic torus $\bG_m$ and hence $D.F=2$.
  We then denote by $D_A$ the sum of all irreducible components of $D$ which are mapped onto $B$ by the Albanese morphism (or ruled fibration) $\alb_X$.
  Namely, $D_A$ is a sum of two cross-sections or a double section of $\alb_X$.
  \item \label{rmk_D_A_3} If $p_g(X)=1$, then let $D_A=0$.
\end{enumerate}

We then claim that in each case above, $X\setminus D_A$ is still a logarithmic Iitaka surface.
Indeed, it suffices to show that $\bar{p}_g(X\setminus D_A) = 1$ since $\bar{\kappa}(X\setminus D_A)\le \bar{\kappa}(X\setminus D) = 0$.
The case (\ref{rmk_D_A_1}) is easy by Lemma \ref{lemma_p_g_bar}.
For the case (\ref{rmk_D_A_2}), our $\bar{p}_g$-formula may not be used due to some unknown invariants.
However, we note that the general fibre of $X\setminus D_A\to B$ is still $\bG_m$ by the definition of $D_A$.
So by Kawamata's addition formula (cf.~\cite[Theorem 11.15]{Iitaka82}),
$\bar{\kappa}(X\setminus D_A)\ge \bar{\kappa}(\bG_m) + {\kappa}(B) = 0$.
Choose $M\in |K_X + D|$ and $N\in |m(K_X + D_A)|$ for some positive integer $m$.
Then by $mM\sim m(K_X + D) \sim N + m(D - D_A)\ge 0$ and $\kappa(X, K_X + D) = 0$, we have $mM = N + m(D-D_A)$.
Thus $M-(D-D_A)=N/m$ is an effective divisor so that $|K_X + D_A|$ is non-empty.
\end{remark}

We keep using the notation $D_A$ till the end of this subsection.
Next, we provide a new and much shorter proof of Iitaka's Theorem III in \cite{Iitaka79}.

\begin{lemma}
\label{lemma_abundance}
Let $V$ be a logarithmic Iitaka surface, and $(X, D)$ a log smooth completion such that $V = X\setminus D$.
Suppose that there is no any $(-1)$-curve in $X\setminus D_A$.
Then $K_X + D_A \sim 0$.
\end{lemma}
\begin{proof}
We only need to show that $K_X + D_A$ is nef.
Indeed, by the claim in Remark \ref{rmk_D_A} we have seen that $\kappa(X, K_X + D_A)=0$ and $\bar{p}_g(X\setminus D_A) = 1$.
Then by the abundance theorem for surfaces, $K_X + D_A \sim_\bQ 0$ (cf.~\cite[Theorem 6.2]{Fujino12}).
In particular, $K_X + D_A \sim 0$ since $\bar{p}_g(X\setminus D_A) = 1$.

Suppose to the contrary that there exists an irreducible curve $C$ such that $(K_X + D_A).C<0$.
Then $C^2<0$ because $\kappa(X, K_X + D_A) = 0$. If $C\nsubseteq D_A$, then $D_A.C\ge 0$ and hence $K_X.C < 0$.
So by the adjunction formula, we know that $C$ is a $(-1)$-curve and $K_X.C = -1$.
Thus $0\le D_A.C < -K_X.C = 1$ implies that $D_A.C = 0$.
This means $C\cap D_A=\varnothing$, contradicting our assumption.
If $C\subseteq D_A$, write $D_A = C + D'_A$ with $D'_A.C\ge 0$.
Then we have
$$(K_X + C).C\le (K_X + C + D'_A).C = (K_X + D_A).C < 0.$$
So by the adjunction formula again, $C$ is a smooth rational curve, i.e., $p_a(C) = 0$.
Hence the cases (\ref{rmk_D_A_1i}) and (\ref{rmk_D_A_2}) in Remark \ref{rmk_D_A} cannot happen since in both cases $p_a(C)\ge 1$.
For the case (\ref{rmk_D_A_1ii}) note that $D'_A.C=2$ since $D_A$ is a cycle of smooth rational curves. Thus
$$0 > (K_X + D_A).C=(K_X + C + D'_A).C=(K_X + C).C + 2 = 2p_a(C) = 0,$$ which is absurd.
For the last case (\ref{rmk_D_A_3}), it is obvious that $C$ is a $(-1)$-curve, which is impossible under our assumption.
\end{proof}

As we mentioned in the remark after Theorem \ref{ThmD}, the natural geometric approach may not apply to bound the dimension.
Hence the following easy observation could be thought as a starting point for proving our main theorem.

\begin{lemma} \label{aut_log_CY}
Let $(X,D)$ be a log smooth pair of dimension $2$ with $V\coloneqq X\setminus D$ such that $K_X + D\sim 0$.
Then $\dim\Aut^0(X,D) = \bar{q}(V)$.
\end{lemma}
\begin{proof}
Let $\Theta_X(-\log D)$ denote the logarithmic tangent sheaf,
which is just the dual of the logarithmic differential sheaf $\Omega_X^1(\log D)$.
Then $H^0(X, \Theta_X(-\log D))$ is the Lie algebra of the connected algebraic group $\Aut^0(X,D)$.
In our situation, $\Omega_X^2(\log D) = \cO_X(K_X + D) = \cO_X$.
Hence $\Theta_X(-\log D) = (\Omega_X^1(\log D))^{\vee}\isom \Omega_X^1(\log D)$.
\end{proof}

\begin{proof}[Proof of Theorem \ref{ThmB}]
We have already seen by Theorem \ref{ThmA} that $\Aut^0(X,D)$ is a semi-abelian variety.
Let $G$ denote the neutral component of $\Aut(X,D_A)$.
Then $\Aut^0(X,D)\le G$ by the special choice of $D_A$ (see Remark \ref{rmk_D_A} for its definition).
Let $$\pi\colon X=X_0 \xrightarrow{\pi_0} X_1 \xrightarrow{\pi_1} \cdots \xrightarrow{\pi_{m-1}} X_m$$
be the composition of morphisms $\pi_i$ such that for every $0\le i\le m-1$,
\begin{enumerate}[(1)]
  \item $\pi_i$ is a blowdown of some $(-1)$-curve $E_i$ in $X_i$,
  \item $D_{A,i+1}\coloneqq {\pi_i}_*D_{A,i}$,
  \item $E_i\subset X_i\setminus D_{A,i}$.
\end{enumerate}
We may assume that $X_m\setminus D_{A,m}$ has no $(-1)$-curve.
Note that $X_m\setminus D_{A,m}$ is still a logarithmic Iitaka surface and the $D_A$-part of $D_{A,m}$ is itself.
Then by Lemma \ref{lemma_abundance}, we have $K_{X_m}+D_{A,m} \sim 0$ and hence
$$\dim\Aut^0(X_m,D_{A,m}) = \bar{q}(X_m\setminus D_{A,m})$$ by Lemma \ref{aut_log_CY}.
Further, it follows from Lemmas \ref{lemma_q_bar} and \ref{lemma_open_subvar} that
$$\bar{q}(X_m\setminus D_{A,m})=\bar{q}(X\setminus D_A)\le \bar{q}(V).$$
Thus we only need to prove that $\pi$ is $G$-equivariant so that $G\le \Aut^0(X_m,D_{A,m})$.
Indeed, the class of each $(-1)$-curve in $\NE(X_i)$ is a $(K_{X_i}+D_{A,i})$-negative extremal ray,
so is preserved by the connected group $G$.
Hence each $\pi_i$ is $G$-equivariant, and so is the composition $\pi$.
\footnote{The $G$-equivariance of the morphism $\pi$ also follows from a general result of Blanchard which asserts that $\pi$ is $G$-equivariant as long as $\pi_*\cO_X = \cO_{X_m}$ (see e.g. \cite[Proposition 4.2.1]{BSU13} for the precise statement).}

From the discussion above, we have
$$\dim\Aut^0(X,D)\le \dim\Aut^0(X,D_A)\le \dim\Aut^0(X_m,D_{A,m}) = \bar{q}(X_m\setminus D_{A,m})\le \bar{q}(V).$$
This completes the proof of Theorem \ref{ThmB}.
\end{proof}

\subsection{Surfaces with \texorpdfstring{$\bar{\kappa} = \bar{p}_g = 0$}{vanishing kappa bar and pg bar}}

Parallel to the previous subsection, we are going to prove Theorem \ref{ThmD} for smooth algebraic surfaces with both logarithmic Kodaira dimension and logarithmic geometric genus vanishing.
Together with logarithmic Iitaka surfaces, they are all smooth algebraic surfaces with vanishing logarithmic Kodaira dimension.

\begin{thm} \label{ThmC}
Let $(X,D)$ be a log smooth pair of dimension $2$ with $V\coloneqq X\setminus D$ such that $\bar{\kappa}(V) = \bar{p}_g(V) = 0$.
Then the neutral component of $\Aut(X,D)$ is a semi-abelian variety of dimension at most $q(X)$.
\end{thm}

Given a smooth algebraic surface $V$ with $\bar{\kappa}(V) = \bar{p}_g(V) = 0$, similarly with Lemma \ref{lemma_log_Iitaka},
we also have some restriction on this surface if it further admits a faithful algebraic $1$-torus action.
Recall that an elliptic ruled surface is a (birationally) ruled surface of genus $1$.

\begin{lemma} \label{p_g_bar_zero}
Let $(X,D)$ be a log smooth pair of dimension $2$ with $V\coloneqq X\setminus D$.
Suppose that $\bar{\kappa}(V) = \bar{p}_g(V) = 0$ and $\bG_m\le \Aut(X,D)$.
Then $X$ is an elliptic ruled surface.
\end{lemma}
\begin{proof}
By a classical characterization of the $\bG_m$-surfaces (i.e., algebraic surfaces admitting algebraic $1$-torus action),
there exists an invariant Zariski open subvariety $U\subseteq V$ equivariantly isomorphic to $C_0\times \bG_m$
with $\bG_m$ acting only on the second factor by translation,
where $C_0$ is a smooth curve (probably non-projective, cf.~\cite[\S1.6, Lemma and \S2.2, Theorem]{OW77}).
Thus there is an equivariant birational map $f\colon X\ratmap C\times \bP^1$ with $C$ the smooth completion of $C_0$.
Since $X$ cannot be a ruled surface of genus $q(X) > 1$ by the same argument as in the proof of Lemma~\ref{lemma_log_Iitaka},
we only need to rule out the case that $X$ is a rational surface.

Suppose to the contrary that $X$ is a rational surface.
Then obviously, the curve $C$ can be taken as $\bP^1$.
So we obtain an equivariant birational map $f\colon X\ratmap Y\coloneqq\bP^1\times \bP^1$.
Note that $\bar{\kappa}(U)\ge \bar{\kappa}(V) = 0$ and hence $C_0 = \bG_m$ or $\bP^1\setminus \{x_1,\ldots,x_t\}$ with $t\ge 3$.

\textsc{Case $1$:} $C_0 = \bG_m$.
Take an equivariant resolution $\pi\colon \widetilde{X}\to X$ of the indeterminacy points of $f$ and $\Sing(X\setminus U)$
such that $\varphi\coloneqq f\circ \pi\colon \widetilde{X}\to Y$ is an equivariant birational morphism
and $\pi^{-1}(X\setminus U)$ is a simple normal crossing divisor.
Since our logarithmic invariants $\bar\kappa$, $\bar{p}_g$ and $\bar q$ are independent of the choice of the log smooth completion (cf.~\cite[\S 11]{Iitaka82}), it follows that
$$\bar{\kappa}(\widetilde{V}) = \bar{\kappa}(V) = 0,\ \bar{p}_g(\widetilde{V}) = \bar{p}_g(V) = 0 \text{ and } \bar{q}(\widetilde{V}) = \bar{q}(V)\le 2,$$
where $\widetilde{V}\coloneqq \widetilde{X}\setminus \widetilde{D}$ with $\widetilde{D}\coloneqq \pi^{-1}(D)$.
Let $\widetilde{B}\coloneqq \pi^{-1}(X\setminus U) \setminus \widetilde{D}$.
We consider the new pair $(\widetilde{X},\widetilde{D}+\widetilde{B})$.
It is easy to see that $\widetilde{X}\setminus (\widetilde{D}+\widetilde{B}) = \pi^{-1}(U) \isom U \isom \bG_m\times \bG_m$ under the birational morphism $\varphi$.
We can take a reduced effective divisor $D_Y\coloneqq 2 \text{ sections} + 2 \text{ fibres}$ on $Y$,
such that $\varphi^{-1}(D_Y) = \widetilde{D}+\widetilde{B}$ since the restriction of $\varphi$ to $\widetilde{X}\setminus (\widetilde{D}+\widetilde{B})$ is an isomorphism.

We have two possibilities according to the dimension of $\varphi(\widetilde{B})$.
Suppose that $\dim \varphi(\widetilde{B}) = 0$, i.e., $\widetilde{B}$ is $\varphi$-exceptional.
Note that $\varphi$ is the composition of blowups of (fixed) points.
Then by Lemma \ref{lemma_q_bar}, we have $\bar{q}(\widetilde{V}) = \bar{q}(Y\setminus D_Y) = 2$.
So by \cite[Corollary~29]{Kawamata81},
the quasi-Albanese morphism $\alpha_{\widetilde{V}}\colon \widetilde{V}\to \cA_{\widetilde{V}}$ is birational,
and hence $\bar{p}_g(\widetilde{V}) = 1$ since $\bar{p}_g$ is a birational invariant for smooth varieties.
This contradicts our assumption.
Suppose that $\dim \varphi(\widetilde{B}) = 1$.
Then $\varphi(\widetilde{D})$ is a proper subset of $D_Y$ with some fibre or section taken away so that $\bar{\kappa}(Y\setminus \varphi(\widetilde{D})) = -\infty$.
In this case, by the logarithmic ramification formula we have
$$K_{\widetilde{X}} + \varphi^{-1}(\varphi(\widetilde{D})) = \varphi^*(K_Y+\varphi(\widetilde{D})) + E,$$
where $E$ is an effective $\varphi$-exceptional divisor (cf.~\cite[Theorem 11.5]{Iitaka82}).
It follows that
$$\bar{\kappa}(\widetilde{X}\setminus \varphi^{-1}(\varphi(\widetilde{D}))) = \bar{\kappa}(Y\setminus \varphi(\widetilde{D})) = -\infty.$$
Also note that $\widetilde{D}\le \varphi^{-1}(\varphi(\widetilde{D}))$
and hence $\bar{\kappa}(\widetilde{V}) = \bar{\kappa}(\widetilde{X}\setminus \widetilde{D}) \le \bar{\kappa}(\widetilde{X}\setminus \varphi^{-1}(\varphi(\widetilde{D}))) = -\infty$,
which is a contradiction.

\textsc{Case $2$:} $C_0=\bP^1\setminus \{x_1,\ldots,x_t\}$ with $t\ge 3$.
The proof of this case is quite similar with the first one.
We keep using the notations there but take $D_Y\coloneqq 2 \text{ sections} + t \text{ fibres}$ on $Y$ with respect to the first projection $p_1\colon \bP^1\times \bP^1 \to \bP^1$,
such that $\varphi^{-1}(D_Y) = \widetilde{D}+\widetilde{B}$.
If $\dim \varphi(\widetilde{B}) = 0$, then by Lemma \ref{lemma_q_bar} we have $\bar{q}(\widetilde{V}) = \bar{q}(Y\setminus D_Y) \ge 3$.
This contradicts the fact that $\bar{q}(\widetilde{V})\le 2$.
If $\dim \varphi(\widetilde{B}) = 1$, we also have two possibilities:
\begin{itemize}
  \item $\varphi(\widetilde{B})$ contains a section: we can derive a contradiction like $\bar{\kappa}(\widetilde{V}) = -\infty$ as in Case $1$.
  \item $\varphi(\widetilde{B})$ consists of fibres: according to the number of the fibres, we can get $\bar{q}(\widetilde{V})\ge 3$,
  or $\bar{q}(\widetilde{V}) = 2$ and $\bar{p}_g(\widetilde{V}) = 1$, or $\bar{\kappa}(\widetilde{V}) = -\infty$.
  All of these cases cannot happen under our assumptions.
\end{itemize}

Therefore, we have proved that $X$ cannot be a rational surface and hence this completes the proof of Lemma~\ref{p_g_bar_zero}.
\end{proof}

\begin{proof}[Proof of Theorem \ref{ThmC}]
First it follows directly from Theorem \ref{ThmA} that $G\coloneqq \Aut^0(X,D)$ is a semi-abelian variety of dimension at most $2$.
If $\dim G = 2$, then by Theorem \ref{ThmA} (\ref{ThmA_2}), $K_X + D\sim 0$ and hence $\bar{p}_g(V)=1$ which is impossible.
So we only to consider the case $\dim G = 1$.
If $G$ is complete, then by the Nishi--Matsumura theorem (cf.~\cite{Matsumura63}),
the induced group homomorphism $G\to \Alb(V)$ has a finite kernel.
In particular, we have $\dim G\le \dim \Alb(V) = q(X)$.
The last remaining case is $G = \bG_m$.
By Lemma \ref{p_g_bar_zero}, $X$ is an elliptic ruled surface with $q(X) = 1$ in this case.
So $\dim G = 1 = q(X)$.
\end{proof}

\subsection{Proof of Theorem \ref{ThmD}}

This follows immediately from Theorems \ref{ThmB} and \ref{ThmC}.

\vskip 1em
\phantomsection
\addcontentsline{toc}{section}{Acknowledgments}
\noindent
\textbf{Acknowledgments. }
The author would like to thank his supervisor Professor De-Qi Zhang for suggesting this problem and many inspiring discussions.
The author also thanks the referee for numerous valuable comments and corrections on an earlier version of this paper. 
The author is supported by a research assistantship of the National University of Singapore.

\end{document}